\documentclass{amsart}
\usepackage{amsmath}
\usepackage{amssymb}
\usepackage{amsthm}
\usepackage[msc-links,abbrev]{amsrefs}
\usepackage{hyperref}
\usepackage[noabbrev,capitalize]{cleveref}
\usepackage{mathrsfs}
\usepackage{mathtools}
\usepackage{slashed}
\usepackage{tikz-cd}
\usepackage{enumitem}

\setenumerate{label=(\roman*)}

\DeclareMathOperator{\im}{im}

\DeclareMathOperator{\sgn}{sgn}

\newcommand{\defn}[1]{{\boldmath\bfseries#1}}

\newcommand{\sha}{\nu}

\newcommand{\om}{\overline{m}}

\newcommand{\comega}{\widetilde{\omega}}

\newcommand{\lv}{\lvert}
\newcommand{\rv}{\rvert}

\newcommand{\mL}{\mathcal{L}}

\newcommand{\bN}{\mathbb{N}}

\newcommand{\bZ}{\mathbb{Z}}

\newcommand{\sA}{\mathscr{A}}

\newcommand{\sR}{\mathscr{R}}

\newcommand{\sW}{\mathscr{W}}

\def\sideremark#1{\ifvmode\leavevmode\fi\vadjust{\vbox to0pt{\vss
 \hbox to 0pt{\hskip\hsize\hskip1em
 \vbox{\hsize3cm\tiny\raggedright\pretolerance10000
 \noindent #1\hfill}\hss}\vbox to8pt{\vfil}\vss}}}

\newcommand{\suchthatcolon}{\mathrel{}:\mathrel{}}

\newtheorem{theorem}{Theorem}[section]
\newtheorem{proposition}[theorem]{Proposition}
\newtheorem{lemma}[theorem]{Lemma}
\newtheorem{corollary}[theorem]{Corollary}

\theoremstyle{definition}

\theoremstyle{remark}

\numberwithin{equation}{section}

\usepackage{tikz-cd}
\usetikzlibrary{arrows}

\begin{document}

\title{A simple construction of the Rumin algebra}
\author{Jeffrey S. Case}
\address{Department of Mathematics \\ Penn State University \\ University Park, PA 16802 \\ USA}
\email{jscase@psu.edu}
\keywords{Rumin complex; Rumin algebra; contact invariant}
\subjclass[2020]{Primary 53D10; Secondary 58A10, 58J10}
\begin{abstract}
 The Rumin algebra of a contact manifold is a contact invariant $C_\infty$-algebra of differential forms which computes the de Rham cohomology algebra.
 We recover this fact by giving a simple and explicit construction of the Rumin algebra via Markl's formulation of the Homotopy Transfer Theorem. 
\end{abstract}
\maketitle

\section{Introduction}
\label{sec:intro}

A \defn{contact manifold} is a pair $(M^{2n+1},\xi)$ of a real $(2n+1)$-dimensional manifold and a subbundle $\xi \subset TM$ of rank $2n$ with the property that if $\theta$ is a local one-form such that $\ker\theta = \xi$, then $\theta \wedge d\theta^n \not=0$.
The \defn{Rumin complex} is a cochain complex which is adapted to the contact structure and computes the de Rham cohomology \emph{groups}.
There are many constructions of the Rumin complex, including Rumin's original construction via subquotients of differential forms~\cite{Rumin1990}, constructions via a spectral sequence~\cites{Rumin2000,Julg1995,BryantEastwoodGoverNeusser2019}, and a realization as a deformation retract of the de Rham complex~\cites{Rumin2005,Case2021rumin,CalderbankDiemer2001}.
Importantly, the latter construction yields, as a consequence of the Homotopy Transfer Theorem~\cites{Kadeishvili1980,LodayVallette2012,Markl2006}, a $C_\infty$-structure on the Rumin complex which recovers the de Rham cohomology \emph{algebra}.

The \defn{Rumin algebra} is a specific contact invariant $C_\infty$-algebra of differential forms which computes the de Rham cohomology algebra;
indeed, its higher products $m_k$, $k \geq 4$, all vanish.
That such a $C_\infty$-algebra exists was first observed by Calderbank and Diemer~\cite{CalderbankDiemer2001} as a specialization of their work on curved Bernstein--Gelfand--Gelfand sequences.
This observation is also implicit in a later approach of Rumin~\cite{Rumin2005} which applies more generally to Carnot--Carath\'eodory structures.
Calderbank--Diemer and Rumin both found a contact invariant deformation retract of the de Rham algebra, but did not explicitly compute the resulting $C_\infty$-structure.
The explicit form of the Rumin algebra was computed by Case~\cite{Case2021rumin} in an ad hoc manner which resembles a construction of Tsai, Tseng and Yau~\cite{TsaiTsengYau2016} on symplectic manifolds.

In this short note we present a clear and simple construction of the Rumin algebra.
The key ingredients are a natural contact invariant deformation retract of the de Rham algebra and an explicit formula~\cite{Markl2006} for the transfer of the Homotopy Transfer Theorem.
This approach gives more explicit formulas for the spaces and deformation retracts identified by Calderbank--Diemer~\cite{CalderbankDiemer2001} and Rumin~\cite{Rumin2005}, and explains the formulas of Case~\cite{Case2021rumin}.

Our construction shows that the Rumin algebra is homotopy equivalent, as a $C_\infty$-algebra, to the de Rham algebra.
Recall that two simply connected manifolds of finite type are rationally homotopy equivalent if and only if their spaces of polynomial differential forms are homotopy equivalent as $C_\infty$-algebras~\cites{FelixHalperinThomas2001,Kadeishvili2009,Sullivan1975,DeligneGriffithsMorganSullivan1975};
and that the homotopy type, as a $C_\infty$-algebra, of the de Rham algebra has many implications for the diffeomorphism type of a closed, simply connected, smooth manifold~\cite{Sullivan1977}.
For these reasons, we expect the Rumin algebra to be a fundamental invariant of contact manifolds.

This note is organized as follows.

In \cref{sec:algebra} we recall some basic facts about $C_\infty$-algebras, including Markl's version~\cite{Markl2006} of the Homotopy Transfer Theorem.

In \cref{sec:contact} we recall some basic facts about contact manifolds and the symplectic structure on their contact distribution.

In \cref{sec:complex} we construct the Rumin algebra.

\section{$C_\infty$-algebras}
\label{sec:algebra}

In this section we discuss some basic facts about $C_\infty$-algebras, which are the commutative versions of $A_\infty$-algebras.
We follow the conventions of Keller's survey article on the latter~\cite{Keller2001}.

Let $A = \bigoplus_{k \in \bZ} A^k$ be a $\bZ$-graded vector space.
A \defn{homogeneous element of $A$} is an element $\omega \in A^k$ for some $k \in \bZ$.
In this case we call $k$ the \defn{degree} of $\omega$ and we set $\lv \omega \rv := k$.
A map $f \colon A^{\otimes k} \to A$ is \defn{homogeneous of degree $\ell$} if
\begin{equation*}
 f( A^{i_1} \otimes \dotsm \otimes A^{i_k} ) \subset A^{i_1 + \dotsm + i_k + \ell}
\end{equation*}
for all $i_1, \dotsc, i_k \in \bZ$.
In this case we set $\lv f \rv := \ell$.

An \defn{$A_\infty$-algebra} is a pair $(A,m)$ of a $\bZ$-graded vector space $A$ and a collection $m = \{ m_k \}_{k\in \bN}$ of homogeneous operators $m_k \colon A^{\otimes k} \to A$ of degree $2-k$ such that
\begin{equation}
 \label{eqn:stasheff-relations}
 \sum_{r+s+t=n} (-1)^{r+st}m_{r+t+1}( 1^{\otimes r} \otimes m_s \otimes 1^{\otimes t}) = 0
\end{equation}
for all $n \in \bN$, where $1 \colon A \to A$ is the identity map and we use the Koszul sign convention
\begin{equation*}
 (f \otimes g)(\alpha \otimes \beta) := (-1)^{\lv g\rv \lv \alpha\rv} f(\alpha) \otimes g(\beta)
\end{equation*}
for homogeneous maps $f,g \colon A \to A$ and homogeneous elements $\alpha,\beta \in A$.
The first three cases of~\eqref{eqn:stasheff-relations} are
\begin{align*}
 m_1m_1 & = 0 , \\
 m_1m_2 & = m_2(m_1 \otimes 1 + 1 \otimes m_1) , \\
 m_1m_3 & = m_2(1 \otimes m_2 - m_2 \otimes 1) - m_3(m_1 \otimes 1^{\otimes 2} + 1 \otimes m_1 \otimes 1 + 1^{\otimes 2} \otimes m_1) .
\end{align*}
In particular, if $(A,m)$ is an $A_\infty$-algebra, then $(A,m_1)$ is a cochain complex.
\defn{Graded (associative) algebras} are $A_\infty$-algebras with $m_k=0$ for all $k\not=2$.
\defn{Differential graded algebras} are $A_\infty$-algebras with $m_k=0$ for all $k \geq 3$.
The \defn{cohomology ring} $(HA,[m_2])$ of $(A,m)$ is the graded algebra
\begin{align*}
 HA^k & := \frac{\ker m_1 \cap A^k}{\im m_1 \cap A^k} , \\
 [m_2]( [\omega_1] \otimes [\omega_2] ) & := [ m_2( \omega_1 \otimes \omega_2 ) ] .
\end{align*}

Given $A_\infty$-algebras $(A,m)$ and $(B,\om)$, an \defn{$A_\infty$-morphism} $f \colon (A,m) \to (B,\om)$ is a collection of homogeneous operators $f_k \colon A^{\otimes k} \to B$ of degree $1-k$ such that
\begin{align*}
 \sum_{r+s+t = n} (-1)^{r+st}f_{r+t+1}( 1^{\otimes r} \otimes m_s \otimes 1^{\otimes t}) & = \sum_{\substack{1 \leq r \leq n \\ i_1 + \dotsm + i_r = n}} (-1)^\ell \om_r (f_{i_1} \otimes \dotsm \otimes f_{i_r} ) , \\
 \ell & := \sum_{j=1}^{r} (r-j)(i_j - 1)
\end{align*}
for all $n \in \bN$.
Specializing to the cases $n\in\{1,2\}$ yields
\begin{align*}
 f_1m_1 & = \om_1 f_1 , \\
 f_1m_2 & = \om_2(f_1 \otimes f_1) + \om_1 f_2 + f_2(m_1 \otimes 1 + 1 \otimes m_1) .
\end{align*}
In particular, $f_1 \colon (A,m_1) \to (B,\om_1)$ is a \defn{cochain map}.
We call $f$ an \defn{$A_\infty$-quasi-isomorphism} if $[f_1] \colon HA \to HB$, $[f_1]([\alpha]) := [f_1(\alpha)]$, is an isomorphism of graded algebras.

Given $p,q \in \bN$, a \defn{$(p,q)$-shuffle} is an element $\sigma \in S_{p+q}$ of the permutation group on $\{ 1, \dotsc, p+q \}$ such that
\begin{equation*}
 \sigma(1) < \dotsm < \sigma(p) , \quad \sigma(p+1) < \dotsm < \sigma(p+q) .
\end{equation*}
Denote by $Sh_{p,q} \subset S_{p+q}$ the subset of $(p,q)$-shuffles.
The \defn{$(p,q)$-shuffle product} $\sha_{p,q} \colon A^{\otimes (p+q)} \to A^{\otimes (p+q)}$ is
\begin{equation*}
 \sha_{p,q}( x_1 \otimes \dotsm \otimes x_{p+q} ) := \sum_{\sigma \in Sh_{p,q}} \sgn(\sigma)\epsilon(\sigma;x_1,\dotsc,x_{p+q}) x_{\sigma^{-1}(1)} \otimes \dotsm \otimes x_{\sigma^{-1}(p+q)} ,
\end{equation*}
where $\epsilon(\sigma;x_1,\dotsc,x_{p+q})$ is the Koszul sign determined by
\begin{equation*}
 x_1 \wedge \dotsm \wedge x_{p+q} = \sum_{\sigma \in S_{p+q}} \epsilon(\sigma;x_1,\dotsc,x_{p+q}) x_{\sigma^{-1}(1)} \otimes \dotsm \otimes x_{\sigma^{-1}(p+q)} 
\end{equation*}
on homogeneous elements.
For example, if $\omega,\tau,\eta \in A$ are homogeneous, then
\begin{align*}
 \sha_{1,1}(\omega \otimes \tau) & = \omega \otimes \tau - (-1)^{\lv\omega\rv\lv\tau\rv} \tau \otimes \omega , \\
 \sha_{1,2}(\omega \otimes \tau \otimes \eta) & = \omega \otimes \tau \otimes \eta - (-1)^{\lv\omega\rv\lv\tau\rv} \tau \otimes \omega \otimes \eta + (-1)^{\lv\omega\rv(\lv\tau\rv + \lv\eta\rv)} \tau \otimes \eta \otimes \omega , \\
 \sha_{2,1}(\omega \otimes \tau \otimes \eta) & = \omega \otimes \tau \otimes \eta - (-1)^{\lv\tau\rv\lv\eta\rv} \omega \otimes \eta \otimes \tau + (-1)^{\lv\eta\rv(\lv\omega\rv + \lv\tau\rv)} \eta \otimes \omega \otimes \tau .
\end{align*}

A \defn{$C_\infty$-algebra} is an $A_\infty$-algebra $(A,m)$ such that
\begin{equation*}
 m_{p+q} \circ \sha_{p,q} = 0
\end{equation*}
for all $p,q \in \bN$.
\defn{Graded commutative (associative) algebras} are $C_\infty$-algebras with $m_k=0$ for all $k\not=2$.
\defn{Commutative differential graded algebras} are $C_\infty$-algebras with $m_k=0$ for all $k\geq3$.

Given $C_\infty$-algebras $(A,m)$ and $(B,\om)$, a \defn{$C_\infty$-morphism} $f \colon (A,m) \to (B,\om)$ is an $A_\infty$-morphism such that
\begin{equation*}
 f_{p+q} \circ \sha_{p,q} = 0
\end{equation*}
for all $p,q \in \bN$.
We call $f$ a \defn{$C_\infty$-quasi-isomorphism} if $[f_1] \colon HA \to HB$ is an isomorphism of graded commutative algebras.

The Homotopy Transfer Theorem~\cite{LodayVallette2012}*{Theorem~10.3.1} constructs a $C_\infty$-structure on a deformation retract of a commutative differential graded algebra and an extension of the inclusion to a $C_\infty$-quasi-isomorphism.
We require explicit formulas for the transferred structure and induced quasi-isomorphism.

\begin{theorem}
 \label{homotopy-transfer}
 Let $(A,d,\mu)$ be a commutative differential graded algebra and let $(B,d)$ be a subcomplex of $(A,d)$.
 Suppose that
 \begin{equation}
  \label{eqn:general-homotopy}
  \begin{tikzcd}
   ( A, d ) \ar[loop left, "h"] \ar[r, shift left, "\pi"] & ( B, d ) \ar[l, shift left, "i"] 
  \end{tikzcd}
 \end{equation}
 is a \defn{deformation retract};
 i.e.\ $\pi \colon (A,d) \to (B,d)$ and $i \colon (B,d) \to (A,d)$ are cochain maps, $h \colon A \to A$ is homogeneous of degree $-1$, and $i \pi = 1_A - dh - hd$ and $\pi i = 1_B$.
 Recursively define $\psi_n \colon A^{\otimes n} \to A$, $n \geq 2$, by
 \begin{equation*}
  \psi_n := \sum_{s + t = n} (-1)^{s+1} \mu( h\psi_{s} \otimes h\psi_{t} ) ,
 \end{equation*}
 with the convention $h\psi_1 = -1_A$.
 Set $m_1 := d$ and $m_k := \pi\psi_k i^{\otimes k}$, $k \geq 2$, and $f_k := -h\psi_k i^{\otimes k}$, $k \in \bN$.
 Then $(B,m)$ is a $C_\infty$-algebra and $f \colon (B,m) \to (A,d,\mu)$ is a $C_\infty$-quasi-isomorphism with $f_1=i$.
\end{theorem}

\begin{proof}
 Markl computed~\cite{Markl2006}*{Theorem~5} that $(B,m)$ is an $A_\infty$-algebra and that $f \colon (B,m) \to (A,d,\mu)$ is an $A_\infty$-quasi-morphism.
 A straightforward induction implies that they are a $C_\infty$-algebra and a $C_\infty$-quasi-isomorphism, respectively.
\end{proof}

\section{The Lefschetz operator on a contact manifold}
\label{sec:contact}

Let $(M^{2n+1},\xi)$ be a contact manifold.
Locally there exists a \defn{contact form};
i.e.\ a real one-form $\theta$ with kernel $\xi$.
We say that $(M^{2n+1},\xi)$ is \defn{coorientable} if a global contact form exists.

Denote by $\sA^k$ the (real) vector space of differential $k$-forms on $M^{2n+1}$ and denote by $\sA_0^k$ the space of \defn{vertical forms};
i.e.\ $\sA_0^k \subset \sA^k$ is the subspace annihilated by taking the exterior product with any local contact form.
We require the following simple observation about the exterior derivative on $\sA_0^k$.

\begin{lemma}
 \label{dsA}
 Let $\theta$ be a local contact form on a contact manifold $(M^{2n+1},\xi)$.
 If $\omega \in \sA_0^k$, then $\theta \wedge d\omega = \omega \wedge d\theta$ wherever $\theta$ is defined.
\end{lemma}

\begin{proof}
 Since $\omega \in \sA_0^k$, it holds that $\omega = \theta \wedge \tau$ for some $\tau \in \sA^{k-1}$.
 Therefore
 \begin{equation*}
  \theta \wedge d\omega = \theta \wedge d(\theta \wedge \tau) = \theta \wedge \tau \wedge d\theta = \omega \wedge d\theta . \qedhere
 \end{equation*}
\end{proof}

Suppose that $(M^{2n+1},\xi)$ is coorientable.
Given a choice of contact form $\theta$, the \defn{Lefschetz operator} $\mL_\theta \colon \sA_0^k \to \sA_0^{k+2}$ is
\begin{equation*}
 \mL_\theta\omega := \omega \wedge d\theta .
\end{equation*}
The restriction to vertical forms ensures that if $u \in C^\infty(M)$ and $\omega \in \sA_0^k$, then
\begin{equation}
 \label{eqn:lefschetz-transformation}
 \mL_{e^u\theta}\omega = e^u \mL_\theta \omega .
\end{equation}

The Lefschetz operator inherits many properties from the symplectic form $d\theta\rv_\xi$.
For example, its powers are isomorphisms when suitably restricted.

\begin{lemma}
 \label{baby-lefschetz}
 Let $(M^{2n+1},\xi)$ be a coorientable contact manifold with global contact form $\theta$.
 If $k \leq n$, then $\mL_\theta^{k} \colon \sA_0^{n-k+1} \to \sA_0^{n+k+1}$ is an isomorphism.
\end{lemma}

\begin{proof}
 Denote by $\sW^k$ the vector space of smooth sections of $\Lambda^k\xi^\ast$.
 Given $\omega \in \sW^k$, denote by $\comega_\theta$ the unique element of $\sA^k$ such that $\comega_\theta\rv_\xi = \omega$ and $\comega_\theta(T_\theta,\cdot) = 0$, where $T_\theta$ is the Reeb vector field determined by $\theta$.
 Then the map
 \begin{equation*}
  \sW^k \ni \omega \mapsto \theta \wedge \comega_\theta \in \sA_0^{k+1}
 \end{equation*}
 is an isomorphism.
 Since $d\theta\rv_\xi$ is a symplectic form, $(d\theta\rv_\xi)^{k} \colon \sW^{n-k} \to \sW^{n+k}$ is an isomorphism~\cite{BryantGriffithsGrossman2003}*{Proposition~1.1}.
 The conclusion readily follows.
\end{proof}

We say that $\omega \in \sA^k$ is \defn{primitive} if $\mL_\theta^{n+1-k}(\theta \wedge \omega) = 0$ for any choice of local contact form.
The following proposition identifies, in a contact invariant way, the non-primitive part of an arbitrary differential form.

\begin{proposition}
 \label{lefschetz}
 Let $(M^{2n+1},\xi)$ be a contact manifold.
 There is a unique contact invariant linear map $\Gamma \colon \sA^k \to \sA_0^{k-1}$ such that given a local contact form $\theta$ and an $\omega \in \sA^k$, it holds that
 \begin{equation}
  \label{eqn:lefschetz}
  \begin{aligned}
   \theta \wedge \omega \wedge d\theta^{n+1-k} & = \Gamma\omega \wedge d\theta^{n+2-k} , && \text{if $k \leq n$} , \\
   \theta \wedge \omega & = \Gamma\omega \wedge d\theta , && \text{if $k \geq n+1$} .
  \end{aligned}
 \end{equation}
 Moreover,
 \begin{enumerate}
  \item $\Gamma(\sA_0^k) = \{ 0 \}$;
  \item $\Gamma d\Gamma = \Gamma$; and
  \item $\Gamma d = 1$ on $\sA_0^k$, $k \leq n$.
 \end{enumerate}
\end{proposition}

\begin{proof}
 We first show the existence, uniqueness, and contact invariance of $\Gamma$.
 
 Let $\theta$ be a local contact form.
 Let $\omega \in \sA^k$.
 If $k \leq n$, then \cref{baby-lefschetz} yields a unique $\zeta_\theta \in \sA_0^{k-1}$ such that
 \begin{equation}
  \label{eqn:low-zeta}
  \theta \wedge \omega \wedge d\theta^{n+1-k} = \zeta_\theta \wedge d\theta^{n+2-k} .
 \end{equation}
 Equation~\eqref{eqn:lefschetz-transformation} implies that $\Gamma_\theta\omega := \zeta_\theta$ is independent of the choice of $\theta$.
 If $k \geq n+1$, then \cref{baby-lefschetz} yields a unique $\zeta_\theta \in \sA_0^{2n-k+1}$ such that
 \begin{equation}
  \label{eqn:high-zeta}
  \theta \wedge \omega = \zeta_\theta \wedge d\theta^{k-n} .
 \end{equation}
 Equation~\eqref{eqn:lefschetz-transformation} implies that $\Gamma_\theta\omega := \zeta_\theta \wedge d\theta^{k-n-1}$ is independent of the choice of $\theta$.
 The existence, uniqueness, and contact invariance of $\Gamma$ readily follow.
 
 We now show the remaining properties of $\Gamma$.
 
 It follows immediately from~\eqref{eqn:low-zeta} and~\eqref{eqn:high-zeta} that $\Gamma(\sA_0^k) = \{ 0 \}$.
 
 Let $\omega \in \sA^k$ and let $\theta$ be a local contact form.
 Suppose first that $k \leq n$.
 Let $\zeta \in \sA_0^{k-1}$ be as in~\eqref{eqn:low-zeta}.
 \Cref{dsA} implies that
 \begin{equation*}
  \theta \wedge d\zeta \wedge d\theta^{n+1-k} = \zeta \wedge d\theta^{n+2-k} .
 \end{equation*}
 Hence $\Gamma d\Gamma = \Gamma$ if $k \leq n$.
 Suppose now that $k \geq n+1$.
 Let $\zeta \in \sA_0^{2n-k+1}$ be as in~\eqref{eqn:high-zeta}.
 \Cref{dsA} implies that
 \begin{equation*}
  \theta \wedge d(\zeta \wedge d\theta^{k-n-1}) = \zeta \wedge d\theta^{k-n} .
 \end{equation*}
 Hence $\Gamma d\Gamma = \Gamma$ if $k \geq n+1$.
 
 Finally, let $\omega \in \sA_0^k$, $k \leq n$.
 \Cref{dsA} implies that
 \begin{equation*}
  \theta \wedge d\omega \wedge d\theta^{n-k} = \omega \wedge d\theta^{n+1-k} .
 \end{equation*}
 Therefore $\Gamma d\omega = \omega$.
\end{proof}

\section{The Rumin algebra}
\label{sec:complex}

Let $(M^{2n+1},\xi)$ be a contact manifold.
Set
\begin{align*}
 \sR^k & := \left\{ \omega \in \sA^k \suchthatcolon \mL^{n+1-k}(\theta \wedge \omega) = 0 , \mL^{n-k}(\theta \wedge d\omega) = 0 \right\} , && \text{if $k \leq n$} , \\
 \sR^k & := \left\{ \omega \in \sA^k \suchthatcolon \theta \wedge \omega = 0 , \theta \wedge d\omega = 0 \right\} , && \text{if $k \geq n+1$} ,
\end{align*}
where $\theta$ is any local contact form.
Equation~\eqref{eqn:lefschetz-transformation} implies that $\sR^k$ is well-defined.
Note that $d(\sR^k) \subseteq \sR^{k+1}$ for all $k \in \bN_0$.
Denote $\sR := \bigoplus_k \sR^k$ and $\sA := \bigoplus_k \sA^k$.

Our main result is that $(\sR,d)$ is a deformation retract of the de Rham complex.

\begin{theorem}
 \label{projection}
 Let $(M^{2n+1},\xi)$ be a contact manifold.
 Then
 \begin{equation}
  \label{eqn:defn-pi}
  \pi\omega := \omega - d\Gamma\omega - \Gamma d\omega
 \end{equation}
 is a homogeneous projection $\pi \colon \sA \to \sR$ of degree zero.
 In particular,
 \begin{equation}
  \label{eqn:homotopy}
  \begin{tikzcd}
   ( \sA, d ) \ar[loop left, "\Gamma"] \ar[r, shift left, "\pi"] & ( \sR, d ) \ar[l, shift left, "i"] 
  \end{tikzcd}
 \end{equation}
 is a deformation retract, where $i \colon \sR \to \sA$ is the inclusion.
\end{theorem}

\begin{proof}
 Equation~\eqref{eqn:defn-pi} implies that $d\pi = \pi d$.
 Hence it suffices to prove that $\pi$ is a projection.
 
 Let $\omega \in \sA$.
 It is clear from~\eqref{eqn:lefschetz} that $\omega \in \sR$ if and only if $\Gamma\omega = \Gamma d\omega = 0$.
 \Cref{lefschetz} implies that $\Gamma \pi = \Gamma d\pi = 0$.
 Therefore $\pi(\sA) \subseteq \sR$ and $\pi^2 = \pi$.
\end{proof}

Applying \cref{homotopy-transfer} to~\eqref{eqn:homotopy} yields an explicit $C_\infty$-structure on the Rumin complex which is quasi-isomorphic to the de Rham algebra.

\begin{corollary}
 \label{structure}
 Let $(M^{2n+1},\xi)$ be a contact manifold.
 Consider $m_k \colon \sR^{\otimes k} \to \sR$,
 \begin{align*}
  m_1 & = di , \\
  m_2 & = \pi \mu i^{\otimes 2} , \\
  m_3 & = \pi \mu ( \Gamma\mu \otimes 1 - 1 \otimes \Gamma \mu )i^{\otimes 3} , \\
  m_k & = 0 , && \text{if $k \geq 4$} .
 \end{align*}
 Then $(\sR,m)$ is a $C_\infty$-algebra.
 Define $f_k \colon \sR^{\otimes k} \to \sA$ by
 \begin{align*}
  f_1 & = i , \\
  f_2 & = -\Gamma\mu i^{\otimes 2} , \\
  f_k & = 0 , && \text{if $k \geq 3$} .
 \end{align*}
 Then $f \colon (\sR,m) \to (\sA,d,\mu)$ is a $C_\infty$-quasi-isomorphism.
\end{corollary}

\begin{proof}
 Since $\Gamma(\sA) \subseteq \sA_0$, it holds that $\mu(\Gamma \otimes \Gamma) = 0$ and $\Gamma \mu (\Gamma \otimes 1) = 0$ and $\Gamma \mu ( 1 \otimes \Gamma) = 0$.
 Combining this with \cref{homotopy-transfer} implies that $(\sR,m)$ is a $C_\infty$-algebra and $f \colon (\sR,m) \to (\sA,d,\mu)$ is a $C_\infty$-quasi-isomorphism.
\end{proof}

The $C_\infty$-algebra $(\sR,m)$ of \cref{structure} is the \defn{Rumin algebra}.
Expanding via the Koszul sign rule recovers the algebra constructed by Case~\cite{Case2021rumin}*{Theorem~8.8}.

\section*{Acknowledgements}
I thank Andreas {\v C}ap for directing me to the work of Calderbank and Diemer~\cite{CalderbankDiemer2001}.
I also thank the University of Washington for providing a productive research environment while I was completing this work.
This work was partially supported by the Simons Foundation (Grant \#524601).

\bibliography{bib}

\newcommand{\noopsort}[1]{}
\begin{bibdiv}
\begin{biblist}

\bib{BryantGriffithsGrossman2003}{book}{
      author={Bryant, Robert},
      author={Griffiths, Phillip},
      author={Grossman, Daniel},
       title={Exterior differential systems and {E}uler-{L}agrange partial differential equations},
      series={Chicago Lectures in Mathematics},
   publisher={University of Chicago Press, Chicago, IL},
        date={2003},
        ISBN={0-226-07794-2},
      review={\MR{1985469}},
}

\bib{BryantEastwoodGoverNeusser2019}{incollection}{
      author={Bryant, Robert~L.},
      author={Eastwood, Michael~G.},
      author={Gover, A.~Rod.},
      author={Neusser, Katharina},
       title={Some differential complexes within and beyond parabolic geometry},
        date={[2019] \copyright 2019},
   booktitle={Differential geometry and {T}anaka theory---differential system and hypersurface theory},
      series={Adv. Stud. Pure Math.},
      volume={82},
   publisher={Math. Soc. Japan, Tokyo},
       pages={13\ndash 40},
      review={\MR{4363266}},
}

\bib{CalderbankDiemer2001}{article}{
      author={Calderbank, David M.~J.},
      author={Diemer, Tammo},
       title={Differential invariants and curved {B}ernstein-{G}elfand-{G}elfand sequences},
        date={2001},
        ISSN={0075-4102},
     journal={J. Reine Angew. Math.},
      volume={537},
       pages={67\ndash 103},
         url={http://dx.doi.org/10.1515/crll.2001.059},
      review={\MR{1856258}},
}

\bib{Case2021rumin}{unpublished}{
      author={Case, Jeffrey~S.},
       title={The bigraded {R}umin complex via differential forms},
        date={\noopsort{2502}preprint},
        note={arXiv:2108.13911},
}

\bib{DeligneGriffithsMorganSullivan1975}{article}{
      author={Deligne, Pierre},
      author={Griffiths, Phillip},
      author={Morgan, John},
      author={Sullivan, Dennis},
       title={Real homotopy theory of {K}\"{a}hler manifolds},
        date={1975},
        ISSN={0020-9910},
     journal={Invent. Math.},
      volume={29},
      number={3},
       pages={245\ndash 274},
         url={https://doi.org/10.1007/BF01389853},
      review={\MR{382702}},
}

\bib{FelixHalperinThomas2001}{book}{
      author={F\'{e}lix, Yves},
      author={Halperin, Stephen},
      author={Thomas, Jean-Claude},
       title={Rational homotopy theory},
      series={Graduate Texts in Mathematics},
   publisher={Springer-Verlag, New York},
        date={2001},
      volume={205},
        ISBN={0-387-95068-0},
         url={https://doi.org/10.1007/978-1-4613-0105-9},
      review={\MR{1802847}},
}

\bib{Julg1995}{article}{
      author={Julg, Pierre},
       title={Complexe de {R}umin, suite spectrale de {F}orman et cohomologie {$L^2$} des espaces sym\'{e}triques de rang {$1$}},
        date={1995},
        ISSN={0764-4442},
     journal={C. R. Acad. Sci. Paris S\'{e}r. I Math.},
      volume={320},
      number={4},
       pages={451\ndash 456},
      review={\MR{1320120}},
}

\bib{Kadeishvili1980}{article}{
      author={Kadeishvili, T.~V.},
       title={On the theory of homology of fiber spaces},
        date={1980},
        ISSN={0042-1316},
     journal={Uspekhi Mat. Nauk},
      volume={35},
      number={3(213)},
       pages={183\ndash 188},
        note={International Topology Conference (Moscow State Univ., Moscow, 1979)},
      review={\MR{580645}},
}

\bib{Kadeishvili2009}{incollection}{
      author={Kadeishvili, Tornike},
       title={Cohomology {$C_\infty$}-algebra and rational homotopy type},
        date={2009},
   booktitle={Algebraic topology---old and new},
      series={Banach Center Publ.},
      volume={85},
   publisher={Polish Acad. Sci. Inst. Math., Warsaw},
       pages={225\ndash 240},
         url={https://doi.org/10.4064/bc85-0-16},
      review={\MR{2503530}},
}

\bib{Keller2001}{article}{
      author={Keller, Bernhard},
       title={Introduction to {$A$}-infinity algebras and modules},
        date={2001},
        ISSN={1532-0081},
     journal={Homology Homotopy Appl.},
      volume={3},
      number={1},
       pages={1\ndash 35},
         url={https://doi.org/10.4310/hha.2001.v3.n1.a1},
      review={\MR{1854636}},
}

\bib{LodayVallette2012}{book}{
      author={Loday, Jean-Louis},
      author={Vallette, Bruno},
       title={Algebraic operads},
      series={Grundlehren der mathematischen Wissenschaften [Fundamental Principles of Mathematical Sciences]},
   publisher={Springer, Heidelberg},
        date={2012},
      volume={346},
        ISBN={978-3-642-30361-6},
         url={https://doi.org/10.1007/978-3-642-30362-3},
      review={\MR{2954392}},
}

\bib{Markl2006}{article}{
      author={Markl, Martin},
       title={Transferring {$A_\infty$} (strongly homotopy associative) structures},
        date={2006},
        ISSN={1592-9531},
     journal={Rend. Circ. Mat. Palermo (2) Suppl.},
      number={79},
       pages={139\ndash 151},
      review={\MR{2287133}},
}

\bib{Rumin1990}{article}{
      author={Rumin, Michel},
       title={Un complexe de formes diff\'{e}rentielles sur les vari\'{e}t\'{e}s de contact},
        date={1990},
        ISSN={0764-4442},
     journal={C. R. Acad. Sci. Paris S\'{e}r. I Math.},
      volume={310},
      number={6},
       pages={401\ndash 404},
      review={\MR{1046521}},
}

\bib{Rumin2000}{article}{
      author={Rumin, Michel},
       title={Sub-{R}iemannian limit of the differential form spectrum of contact manifolds},
        date={2000},
        ISSN={1016-443X},
     journal={Geom. Funct. Anal.},
      volume={10},
      number={2},
       pages={407\ndash 452},
         url={https://doi.org/10.1007/s000390050013},
      review={\MR{1771424}},
}

\bib{Rumin2005}{article}{
      author={Rumin, Michel},
       title={An introduction to spectral and differential geometry in {C}arnot-{C}arath\'{e}odory spaces},
        date={2005},
        ISSN={1592-9531},
     journal={Rend. Circ. Mat. Palermo (2) Suppl.},
      number={75},
       pages={139\ndash 196},
      review={\MR{2152359}},
}

\bib{Sullivan1975}{inproceedings}{
      author={Sullivan, D.},
       title={Differential forms and the topology of manifolds},
        date={1975},
   booktitle={Manifolds---{T}okyo 1973 ({P}roc. {I}nternat. {C}onf., {T}okyo, 1973)},
       pages={37\ndash 49},
      review={\MR{0370611}},
}

\bib{Sullivan1977}{article}{
      author={Sullivan, Dennis},
       title={Infinitesimal computations in topology},
        date={1977},
        ISSN={0073-8301},
     journal={Inst. Hautes \'{E}tudes Sci. Publ. Math.},
      number={47},
       pages={269\ndash 331 (1978)},
         url={http://www.numdam.org/item?id=PMIHES_1977__47__269_0},
      review={\MR{646078}},
}

\bib{TsaiTsengYau2016}{article}{
      author={Tsai, Chung-Jun},
      author={Tseng, Li-Sheng},
      author={Yau, Shing-Tung},
       title={Cohomology and {H}odge theory on symplectic manifolds: {III}},
        date={2016},
        ISSN={0022-040X},
     journal={J. Differential Geom.},
      volume={103},
      number={1},
       pages={83\ndash 143},
         url={http://projecteuclid.org/euclid.jdg/1460463564},
      review={\MR{3488131}},
}

\end{biblist}
\end{bibdiv}
\end{document}